\documentclass[11pt]{amsart}

\usepackage{amssymb}
\usepackage{amsfonts}
\usepackage{amsmath}
\usepackage{amsthm}
\usepackage{graphicx}
\usepackage[all]{xy}
\usepackage[]{hyperref}
\usepackage{color}
\usepackage{tikz}
\usetikzlibrary{patterns}
\usepackage{tikz-3dplot}
\usepackage{caption}
\usepackage{subcaption}

\usepackage{fancyhdr}

\DeclareMathOperator{\CAT}{CAT}
\DeclareMathOperator{\Vol}{Vol}

\DeclareMathOperator{\PSL}{PSL}

\usepackage{color}
\newcounter{comhar}
\newcounter{comdave}
\begin{document}

\title[Entropy rigidity for convex projective manifolds]{Entropy rigidity for finite volume
strictly convex projective manifolds}

\author{
Harrison Bray and David Constantine
}
\date{\today}

\maketitle

\theoremstyle{plain} \newtheorem{thm}{Theorem}[section]
\theoremstyle{plain} \newtheorem{theorem}[thm]{Theorem}
\theoremstyle{plain} \newtheorem{conj}[thm]{Conjecture}
\theoremstyle{plain} \newtheorem{lemma}[thm]{Lemma}
\theoremstyle{plain} \newtheorem{prop}[thm]{Proposition}
\theoremstyle{plain} \newtheorem{proposition}[thm]{Proposition}
\theoremstyle{plain} \newtheorem{cor}[thm]{Corollary}	
\theoremstyle{plain} \newtheorem{corollary}[thm]{Corollary}	
\theoremstyle{definition} \newtheorem{defn}[thm]{Definition}
\theoremstyle{definition} \newtheorem{definition}[thm]{Definition}	
\theoremstyle{remark} \newtheorem{rmk}[thm]{Remark}
\theoremstyle{remark} \newtheorem{remark}[thm]{Remark}
\theoremstyle{remark} \newtheorem{obs}[thm]{Observation}
\theoremstyle{remark} \newtheorem*{quest}{Question}

\begin{abstract}
We prove entropy rigidity for finite volume strictly convex projective manifolds in dimensions $\geq 3$, generalizing the work of \cite{abc} to the finite volume setting. The rigidity theorem uses the techniques of Besson, Courtois, and Gallot's entropy rigidity theorem. It implies uniform lower bounds on the volume of any finite volume strictly convex projective manifold in dimensions $\geq 3$.
\end{abstract}


%

\section{Introduction}

In this note we prove the following theorem, which generalizes Theorem 1.10 of \cite{abc} to the finite volume setting:

\begin{thm}\label{thm:main}
Let $Y_\Omega$ be a finite volume, strictly convex projective manifold of dimension $n\geq 3$, equipped with its Hilbert metric. Suppose that $Y_0$ is a hyperbolic structure on the same manifold. Then there is a number $N(F_\Omega)\geq 1$ such that
\[ N(F_\Omega) h(F_\Omega)^n \Vol(Y, F_\Omega) \geq h(g_0)^n \Vol(Y, g_0).\]
Furthermore, equality holds if and only if $(Y, F_\Omega)$ is isometric to $(Y, g_0)$.
\end{thm}

Here, $h$ refers to volume growth entropy of metric balls in the universal cover with respect to the given metric. While interesting in its own right, Theorem \ref{thm:main} has a corollary of particular geometric interest.

\begin{cor} \label{cor:uniform_volume_bound}
There is a constant $\mathcal{D}>0$, depending only on dimension, such that if $Y$ is a finite volume, strictly convex projective manifold of dimension at least three which also admits a hyperbolic structure $g_0$, then
\[ \Vol(Y,F_\Omega) \geq \mathcal{D} \Vol(Y, g_0).\]
\end{cor}

In particular, $\Vol(Y,g_0)$ is a constant in dimension $n\geq 3$ by Mostow-Prasad rigidity \cite{prasad}, and hence we obtain a universal lower bound on volume.

In \cite{abc} a second corollary of Theorem \ref{thm:main} was that if $Y_\Omega$ is deformed so that $h(F_\Omega)\to 0$, then $\Vol(Y,F_\Omega) \to \infty$. This no longer applies in the finite volume, non-compact setting for the following reason. In the setting of finite volume strictly convex projective manifolds, the volume growth entropy agrees with the critical exponent of the Poincar\'e series for the action of the group on the universal cover \cite[Th\'eor\`eme 9.2]{cramponmarquis14}. Moreover, for the case of an $n$-manifold, the critical exponent is bounded below by $\frac{n-1}2$ (see \cite[Lemme 9.4]{cramponmarquis14}, or for a proof in english, the case $n=2$ is proven in \cite[Lemma 4.3.3]{cramponthese} and generalizes to higher dimensions) which is the critical exponent of a maximal rank parabolic subgroup acting on hyperbolic $n$-space. Thus, it is not possible for the volume growth entropy to be arbitrarily small, unlike what is known to occur in some cocompact examples in small dimensions (see \cite[Theorem 1.4 \& Corollary 1.6]{nie}, or \cite[Corollary 3.7]{zhang_degeneration} for arbitrary surfaces).  It is possible that for any family of representations for which the volume growth entropy converges to $\frac{n-1}{2}$, then the volume of the quotient must diverge to infinity. There is some experimental evidence of this behavior in dimension 2, and the volume does diverge to infinity with the entropy\footnote{see experiemental work and images generated by Marianne DeBrito, Andrew Nguyen, and Marisa O'Gara as a part of the LoGM program at the University of Michigan here: \url{https://gitlab.eecs.umich.edu/logm/wi20/entropy-project-outputs}}, but the result does not immediately follow Theorem \ref{thm:main}.

\subsection{Background}

Theorem \ref{thm:main} is inspired by the work of Besson, Courtois, and Gallot on entropy rigidity in the Riemannian setting \cite{BCG1, BCG2}. They prove that on a compact manifold supporting a negatively-curved locally symmetric metric, that metric strictly minimizes $h(g)^n\Vol(Y,g)$ among all Riemannian metrics $g$. Their method of proof -- the \emph{barycenter method} -- is now a fundamental tool with far-reaching impact; it will be the method of this paper. Their work was extended to finite volume Riemannian metrics in \cite{BCS} and \cite{storm}; the arguments of \cite{BCS} accomplishing this extension are the model for the present work. A nice survey of the barycenter method can be found in \cite{CFrecent}.

Strictly convex projective manifolds are a natural place to look for analogues of rigidity theorems which rely on the dynamics and geometry of negatively curved spaces. Hyperbolic manifolds are the first examples of strictly convex projective manifolds, realized as quotients of the Beltrami-Klein model of hyperbolic space, but unlike for hyperbolic geometry on $n\geq 3$ manifolds, the deformation space of strictly convex projective structures on a manifold can be nontrivial. Most relevant to this work, there are examples in every dimension of nontrivial moduli spaces of strictly convex projective structures of finite volume, which arise as deformations of the hyperbolic model via the Johnson-Millson bending construction \cite{johnsonmillson}, constructed by \cite{ballasmarquis, marquis12}. There are other known deformable examples in dimension 3 that arise from a generalization of Thurston's gluing equations \cite[Theorem 0.4]{ballascasella}.

These nonhyperbolic strictly convex projective structures admit a Finsler geometry via the Hilbert metric, which retains some but not all traits from hyperbolic space. For example, strict convexity of the geometry is equivalent to rel hyperbolicity of the fundamental group, and hence Gromov-hyperbolicity of the metric space, when the action has finite covolume \cite[Theorem 0.15]{clt}. On the other hand, the Hilbert geometry is not even $\CAT(0)$ in general. Though strictly convex projective structures are coarsely hyperbolic, our assumption that the quotient admits a hyperbolic metric is nontrivial and presumably necessary; there are closed manifolds in any dimension greater than three which admit strictly convex projective structures, but do not admit a hyperbolic metric \cite{Ben5,Kapovich2007}. 

The Hilbert metric is compatible with a projectively invariant Hausdorff measure called the Busemann-Hausdorff measure which determines our notion of finite volume, though we note that finite volume is well-defined for any choice of projectively invariant volume measure and refer the reader to the nice survey of Marquis for more details \cite{marquis_handbook}. (See \cite{Hilberthandbook} and the essays therein for other fundamentals and a survey of the area.)

Together with Adeboye, the authors proved Theorem \ref{thm:main} and its corollaries for closed, strictly convex projective manifolds in \cite{abc}. As in \cite{abc}, the particularly nice geometric properties of the Hilbert metric simplify some portions of the proofs when compared with the general Riemannian case. In particular, the equality case of Theorem \ref{thm:main} has a much simpler proof in our setting, and the geometry of the cusps is much more well-controlled, which allows us to mimic the argument of \cite{BCS} when extending to finite volume, rather than the more complicated but more general argument of \cite{storm}.

%

\subsection{Outline of the paper}

In Section \ref{sec:prelim} we define the basic objects involved in the argument and collect a few important facts about them. In Section \ref{sec:geometric lemmas}, we prove a number of geometric lemmas which will be used in the proof.  Section \ref{sec:bcg} contains a quick review of the Besson--Courtois--Gallot argument and proofs of our main theorems, conditional on a specific map between $(Y,F_\Omega)$ and $(Y,g_0)$ being proper. Showing that this `natural map' is proper is the key step in moving from the compact to finite-volume setting and the main work of this paper; it is carried out in Section \ref{sec:proper}. Finally, in Section \ref{sec:main} we complete the proof of Theorem \ref{thm:main} and its corollaries.

%

\subsection{Acknowledgements}

We would like to thank Dick Canary and Ralf Spatzier for helpful conversations about this project. D.C. would like to thank the University of Michigan for hosting him on a short visit during which a portion of this work was completed.

%

\section{Preliminaries}\label{sec:prelim}

Let $\Omega$ be a properly convex domain in $\mathbb{RP}^n$, meaning there exists an affine chart in which $\Omega$ is bounded. Then $\Omega$ is strictly convex if in such an affine chart, the intersection of its topological boundary $\partial \Omega$ with any line in the complement of $\Omega$ is at most one point. The Hilbert metric on any properly convex domain is defined as follows; chose an affine chart in which $\Omega$ is bounded, and for any $x,y\in\Omega$, take $a,b$ to be the intersection points in $\partial \Omega$ of any projective line containing $x$ and $y$. Then
\[
  d_\Omega(x,y)=\frac{1}{2} \left| \log [a:x:y:b] \right|
\]
where $[a:x:y:b]=\frac{|a-y||b-x|}{|a-x||b-y|}$ denotes the Euclidean cross-ratio in an affine chart, a projective invariant.

Let $F_\Omega$ be the Finsler metric on $\Omega$ induced by the Hilbert distance. This metric induces a projectively invariant volume form -- the Hilbert volume -- and we will denote volumes computed with this volume form by $\Vol(-,F_\Omega)$. (See \cite{abc} for the definition of this volume.)

Suppose that $\Gamma$ acts freely and properly discontinuously on $\Omega$ by projective transformations, which are isometries of the Hilbert metric, so that $Y=\Omega/\Gamma$ is a manifold with $\Vol(Y,F_\Omega)<+\infty$. In \cite{abc} the case where $Y$ is compact was handled, so we assume throughout that $Y$ is finite volume but not compact.

Although we work specifically with the Hilbert volume, we note that it is a straightforward consequence of Benz\'ecri's compactness theorem any two projectively invariant volumes on $\Omega$ are equivalent (up to bounded multiplicative bounds, see \cite[Prop. 9.4]{marquis_handbook}). In particular (see \cite[Cor. 9.5]{marquis_handbook}), finiteness of the volume of $Y$ does not depend on the specific choice of volume.

We assume further that $Y$ supports a hyperbolic metric $g_0$. We will denote the Hilbert and hyperbolic structures on $Y$ by $Y_\Omega=(Y, F_\Omega)$ and $Y_0=(Y, g_0)$, respectively.  As the underlying manifold for these two structures is the same, there is a homeomorphism between them; in particular this is a \emph{proper} map. Let us denote this map by $f:Y_\Omega \to Y_0$. We mark the universal covers of these spaces and lifts of various objects to these universal covers using $\sim$'s. For example, $\tilde f\colon \tilde Y_\Omega \to \tilde Y_0$ is the $\Gamma$-equivariant lift of $f$.

$\partial_\infty\tilde Y_-$ denotes the boundary at infinity. We note that $\tilde f$ extends to a $\Gamma$-equivariant homeomorphism between $\partial_\infty\tilde Y_\Omega$ and $\partial_\infty\tilde Y_0$, which we also denote by $\tilde f$. We will denote a cusp of $Y_-$ by $\Theta$ and use $\tilde \Theta$ to refer to a point in $\partial_\infty\tilde Y_-$ to which that cusp is asymptotic when lifted to the universal cover.

For any closed, convex set $C$ in $\tilde Y_0$ and any point $x\in \tilde Y_0\smallsetminus C$, let $v(x,C)$ be the unit tangent vector based at $x$ which is tangent to the unique distance-minimizing geodesic from $x$ to $C$. We write $v(x,\xi)$ instead of $v(x,\{\xi\})$ when $\xi$ is a point; in this case we allow $\xi$ to lie in $\tilde Y_0\cup\partial_\infty\tilde Y_0$.

Two families of measures on $\partial_\infty\tilde Y_\Omega$ and $\partial_\infty\tilde Y_0$ are at the heart of Besson, Courtois, and Gallot's approach to entropy rigidity. The more simple is the family of visual measures. For any $y\in \tilde Y_0$, let $\exp: T_y^1\tilde Y_0 \to \partial_\infty \tilde Y_0$ send a unit tangent vector to the endpoint at infinity of the geodesic ray it generates.

\begin{definition}\label{def:visualmeasure}
To each $y\in \tilde Y_0$, we associate the {\em visual measure} $\nu_y$ on $\partial_\infty \tilde Y_0$ by pushing forward under exp the Hausdorff measure $\sigma_y$ induced by the angular metric on $T_y^1 \tilde Y_0$.
\end{definition}

On $\partial_\infty \tilde Y_\Omega$, the \emph{Patterson-Sullivan} measures play a central role. Before introducing them we must define another central object for our arguments.

\begin{definition}\label{def:busemannfunctions}
Given $o\in \tilde Y_-$ and $\alpha \in \partial_\infty \tilde Y_-$, the \emph{Busemann function centered at $\alpha$ and based at $o$} is $B_{o,\alpha}:\tilde Y_-\to \mathbb{R}$ defined by
\[ B_{o,\alpha}(y) = \lim_{t \to \infty} d(y,c(t)) - d(o,c(t))\]
where $c(t)$ is any geodesic ray heading to $\alpha$. We specifically denote Busemann functions in $\tilde Y_\Omega$ by $B^\Omega$ and those in $\tilde Y_0$ by $B^0$. A {\em horosphere} centered at a point at infinity $\alpha$ is a level set for the Busemann function $B^-_{o,\alpha}(\cdot)$, and a {\em horoball} centered at $\alpha$ is the convex interior of a horosphere.

\end{definition}
Note that $B^-_{o,\alpha}(y)$ is 1-Lipschitz in $y$ with gradient $d B^-_{o,\alpha}(y) = - v(y,\alpha)$.

\begin{definition}\label{def:pattersonsullivan}
The \emph{Patterson-Sullivan density} is an assignment $x \mapsto \mu_x$ of a finite measure on $\partial_\infty\tilde Y_\Omega$ to each point in $\tilde Y_\Omega$ satisfying the following two properties:
\begin{itemize}
	\item (quasi-$\Gamma$-invariance) $\mu_{\gamma x} = \gamma_* \mu_x$ for all $\gamma\in \Gamma$ and all $x\in \tilde Y_\Omega$, and
	\item (transformation rule) $\frac{d\mu_x}{d\mu_y}(\beta) = e^{-h B^\Omega_{y,\beta}(x)}$
\end{itemize}
where $h=\delta_\Gamma$ is the critical exponent for the action of $\Gamma$ on $\tilde Y_\Omega$ or, equivalently by \cite[Theorem 1.11]{cramponmarquis14}, the volume growth entropy of $(\tilde Y_\Omega, F_\Omega)$.
\end{definition}

In the setting of strictly convex real projective structures on finite volume manifolds, Crampon constructs the Patterson-Sullivan measures explicitly so that they have full support on the limit set of $\Gamma$ \cite[Section 4.2.1]{cramponthese}, which is equal to the boundary of $\Omega$ in the cofinite case \cite[Corollaire 1.5]{cramponmarquis_finitude}. Crampon provides a proof for surfaces that the critical exponent of the group $\Gamma$ acting on $\Omega$ with the Hilbert metric is strictly larger than the critical exponent of a parabolic group acting on the hyperbolic plane \cite[Lemma 4.3.3]{cramponthese}. The proof is a standard ping-pong argument which extends to higher dimensions by \cite[Corollaire 7.18]{cramponmarquis_finitude}, a generalization of Crampon's \cite[Lemma 1.3.4]{cramponthese}. Thus, in our setting, the critical exponent is strictly larger than the critical exponent of a maximal rank parabolic subgroup acting on $\mathbb H^n$, which is constant equal to $\frac{n-1}2$. As an application, one can extend Crampon's argument in \cite[Proposition 4.3.5]{cramponthese} to show that the Patterson-Sullivan measures have no atoms.

Note that the results of Crampon and Crampon-Marquis assuming both strict convexity of $\Omega$ and that the boundary of $\Omega$ is $C^1$ still apply, because Cooper-Long-Tillman proved that these properties are equivalent when the action is cofinite \cite[Theorem 0.15]{clt}.

By $\mathcal{M}(\partial_\infty \tilde Y_-)$ we denote the set of all finite measures on $\partial_\infty\tilde Y_-$. Throughout, $\Vert \lambda \Vert := \lambda(\partial_\infty \tilde Y_-)$ will denote the total mass of $\mu$.

\begin{definition}\label{def:barycenter}
Fix some basepoint $o\in\tilde Y_0$. For any finite measure $\lambda$ on $\partial\tilde Y_0$ and $y\in\tilde Y_0$, let
\[  \mathcal B(y,\lambda):=\int_{\alpha\in\partial\tilde{X}} B_{o,\alpha}^0(y)d\lambda(\alpha). \]
$B_{o,\alpha}^0$ is convex along geodesic segments and strictly convex along segments which do not have an endpoint at $\alpha$, hence $\mathcal B(y,\lambda)$ has a unique minimum for non-atomic $\lambda$ \cite[Appendix A]{BCG1}. Denote this minimum by $bar(\lambda)$; this is the \emph{barycenter} of $\lambda$.
It is a straightforward exercise to check that the barycenter of $\lambda$ is $\Gamma$-equivariant:
\[
  bar(\gamma_*\lambda)=\gamma\cdot bar(\lambda) \text{ for all } \gamma\in\Gamma.
\]
\end{definition}

%

\section{Geometric lemmas}\label{sec:geometric lemmas}

%

\subsection{Cusps in convex projective manifolds}

The arguments in this paper which go beyond those in \cite{abc} are primarily about the cusps of the finite volume but noncompact Hilbert geometry. As in hyperbolic geometry, a {\em cusp} is a small neighborhood of a boundary component in the manifold, and its holonomy group is called the {\em cusp group}. The boundary component of a cusp is called an {\em end}. In the finite volume case, all cusps are {\em maximal rank cusps}, meaning the boundary component is compact.  A nice simplifying feature in this setting is the following description of the cusps:

\begin{theorem}[{\cite[Theorem 0.4]{clt}, \cite[Theorem 0.5]{clt}, \cite[Theorem 1.7]{cramponmarquis_finitude}}]\label{thm:thickthin}
Every maximal rank cusp in a strictly convex real projective manifold is projectively equivalent to a hyperbolic cusp of the same dimension, meaning the cusp holonomies are conjugate.
\label{thm:equiv_cusps}
\end{theorem}

In particular, the end of a maximal rank cusp is a single point, and Theorem \ref{thm:equiv_cusps} implies the stabilizing cusp group is virtually $\mathbb Z^{n-1}$, where $n$ is the dimension of the manifold.

On the level of the universal cover, each lift of an end is a {\em bounded parabolic point}, which is a point in the boundary of the universal cover whose stabilizer contains only parabolic isometries, and preserves and acts cocompactly on each horosphere centered at the fixed bounded parabolic point \cite[Proposition 5.6]{clt}, \cite[Th\'eor\`eme 3.3]{cramponmarquis_finitude}. Parabolic isometries are defined as elements of the group for which the infimum of the displacement of a point with respect to the Hilbert metric is equal to zero and not realized. Though the notion of parabolic point is more general than the notion of bounded parabolic point, there is no need for this distinction in our setting, hence we will use the term {\em parabolic point} to refer throughout to the lift of an end in the quotient to the universal cover.  A stabilizer of a parabolic point is called a {\em parabolic group}.

A more explicit result implying Theorem \ref{thm:equiv_cusps} is the following:

\begin{theorem}[{\cite[Th\'eor\`eme 7.14]{cramponmarquis_finitude}}]
Let $\Omega$ be a properly, strictly convex domain in $\mathbb RP^n$ with $C^1$ boundary, and let $P$ be a maximal rank parabolic subgroup of $\PSL(n+1,\mathbb R)$ preserving $\Omega$ which fixes the boundary point $p$. Then there exist $P$-invariant {\em osculating ellipsoids} to $\Omega$ at $p$, denoted $\mathcal E_{in}$ and $\mathcal E_{out}$, meaning 
  \begin{itemize}
    \item $\mathcal E_{in}\subset\Omega\subset \mathcal E_{out}$,
    \item $\partial\mathcal E_{in}\cap\partial\Omega=\partial \mathcal E_{out}\cap\partial\Omega = \{p\}$,
    \item $\mathcal E_{in}$ is a horoball in $\mathcal E_{out}$ endowed with the Hilbert metric.
  \end{itemize}
  \label{thm:ellipsoids}
\end{theorem}

With this theorem we easily establish the following fact, which will be
employed later:

\begin{lemma}\label{lem:short loops in cusps}
For any $\epsilon_0>0$, any parabolic point $\tilde \Theta$ in $\partial_\infty \tilde Y_\Omega$ with stabilizer $\Gamma_{\tilde\Theta}$, and any finite set of elements $p_1,\ldots,p_{n-1}\in \Gamma_{\tilde\Theta}$, there exists an open horoball $U$ centered at $\tilde\Theta$ such that if $x\in U$ then $d_\Omega(x,p_ix)<\epsilon_0$ for all $i=1,\ldots,n-1$. 
\end{lemma}

\begin{proof}
Let $\Omega$ be a properly convex domain in $\mathbb RP^n$, and represent $\Gamma$ as a discrete group of projective transformations acting cofinitely on $\Omega$ such that the quotient $\Omega/\Gamma$ is isometric to $Y_\Omega$ when endowed with the Hilbert metric (and hence $\Omega$ is isometric to $\tilde Y_\Omega$).  By Theorem \ref{thm:ellipsoids}, there is a $\Gamma_{\tilde\Theta}$-invariant ellipsoid $\mathcal E$ contained in $\Omega$, and tangent to $\Omega$ at $\tilde \Theta$. Since $\mathcal E$ with the Hilbert metric is isometric to hyperbolic $n$-space, for all $\epsilon>0$ there is a horoball $H^{\mathcal E}$ centered at $\tilde \Theta$ in the metric space $(\mathcal E,d_{\mathcal E})$ such that for all $x\in H^{\mathcal E}$ and $i=1,\ldots n-1$, we have $d_{\mathcal E_1}(x,p_ix)<\epsilon$.

Since $\Gamma_{\tilde\Theta}$ preserves both $\Omega$ and $\mathcal E$ as a parabolic group, $\Gamma_{\tilde\Theta}$ preserves and acts cocompactly on horospheres for the metric spaces $(\Omega,d_\Omega)$ and $(\mathcal E,d_{\mathcal E})$. Thus since $\mathcal E$ is a subset of $\Omega$, there exists a horoball $H^\Omega$ for the metric space $(\Omega,d_\Omega)$ which is contained in $H^{\mathcal E}$. 

Lastly, it is a straightforward observation using the cross-ratio that if $\mathcal E\subset \Omega$, then $d_{\mathcal E}\geq d_{\Omega}$ when restricted to $\mathcal E$. Thus for all $x\in H^\Omega\subset H^{\mathcal E}$ and all $i=1,\ldots,n-1$, 
  \[
    d_\Omega(x,p_i x)\leq d_{\mathcal E}(x,px) <\epsilon.
  \]
\end{proof}

%

\subsection{Barycenters of visual and Patterson-Sullivan measures}

In this section we first prove some general lemmas about the barycenter map and the families of visual and Patterson-Sullivan measures. Then we prove some specific results on the behavior of these objects for points in the cusp of one of our manifolds.

Below we will denote a closed half-space in the hyperbolic space $\tilde Y_0$ by $H$.  The hyperplane boundary of $H$ in $\tilde Y_0$ is denoted $\partial H$. The boundary at infinity of $H$ is denoted by $\partial_\infty H$.

\begin{lemma}\label{lem:bar visual}
For any $y\in \tilde Y_0$, $bar(\nu_y) = y$.
\end{lemma}

\begin{proof}
  Fix a basepoint $o\in \tilde Y_0$. By Definition \ref{def:barycenter}, $bar(\nu_y)$ occurs at the unique point $y' \in Y$ where $d\mathcal{B}(y',\nu_y)=0$. A simple calculation using the fact that $dB^0_{o,\alpha}(y')=-v(y',\alpha)$ proves that this happens when $y'=y.$
\end{proof}

The following Lemma is similar in spirit to \cite[Lemma 3.2]{BCS}.

\begin{lemma}\label{lem:control_bar}
There is a uniform constant $D>0$ such that the following holds. If $\lambda \in \mathcal{M}(\partial_\infty \tilde Y_0)$, $bar(\lambda)$ is defined, and $H$ is a closed halfspace in $\tilde Y_0$ such that $\lambda(\partial_\infty H) >\frac{2}{3}\Vert \lambda \Vert$, then $d_{g_0}(bar(\lambda),H)\leq D$.
\end{lemma}

\begin{proof}
First, for the hyperbolic space $\tilde Y_0$, it is clear that there exists a constant $D$ such that for any $\alpha\in \partial_\infty H$, if $d_{g_0}(y,H)>D$ then $\langle v(y,H),v(y,\alpha) \rangle>\frac{1}{2}$.  (Here $\langle v_1, v_2\rangle = g_0(v_1,v_2)$.) Suppose $\lambda(\partial_\infty H^+)>\frac{2}{3}\Vert \lambda\Vert$ and let $y$ be any point with $d_{g_0}(y,H)>D$. Then
\begin{multline*}
	\langle v(y,H), -d\mathcal{B}(y,\lambda) \rangle = \int_{\alpha\in \partial \tilde Y_0} \left\langle v(y,H), v(y,\alpha) \right\rangle d\lambda(\alpha)  \\ 
		=  \int_{\alpha\in \partial_\infty H} \left\langle v(y,H), v(y,\alpha) \right\rangle d\lambda(\alpha) + \int_{\alpha\notin \partial_\infty H} \left\langle v(y,H), v(y,\alpha) \right\rangle d\lambda(\alpha) \\ 
		> \frac{1}{2} \lambda(\partial_\infty H) - (\Vert\lambda\Vert-\lambda(\partial_\infty H))   >0,
\end{multline*}
since $\lambda(\partial_\infty H)>\frac{2}{3}\Vert \lambda\Vert$. Since $d\mathcal{B}(y,\lambda) \neq 0$, $y$ cannot be $bar(\lambda)$.
\end{proof}

\begin{lemma}\label{lem:PS rule}
For all measurable $A\subset \partial_\infty \tilde Y_\Omega$,
\[ e^{-h d_\Omega(x,\gamma x)} \mu_x(A) \leq \mu_x(\gamma A) \leq e^{h d_\Omega(x,\gamma x)} \mu_x(A). \]
\end{lemma}

\begin{proof}
The transformation rule of the Patterson-Sullivan family and the 1-Lipschitz property of Busemann
functions (Definitions \ref{def:pattersonsullivan} and \ref{def:busemannfunctions}) together imply 
\[ 
	e^{-hd_\Omega(x,\gamma x)}\leq \frac{d\mu_{\gamma^{-1}x}}{d\mu_x}(\beta) \leq
	 e^{hd_\Omega(x,\gamma x)}.
\]
The result now follows from the quasi-$\Gamma$-invariance of the Patterson-Sullivan measures. 
\end{proof}

Key in our arguments in Section \ref{sec:proper} will be control of the visual and Patterson-Sullivan measures for points in a cusp. These are provided by the following Lemmas.

\begin{lemma}\label{lem:control_visual}
Let $H$ be any open halfspace in the hyperbolic space $\tilde Y_0$ such that $\partial_\infty H$ contains $\tilde \Theta$. Then there exists a horoball $B'_1 \subset \tilde Y_0$ based at $\tilde \Theta$ such that for all $y \in B'_1$, $\nu_y(\partial_\infty H) > \frac{2}{3}\Vert \nu_y\Vert$.
\end{lemma}

\begin{proof}
Consider the upper-halfspace model for $\tilde Y_0\cong \mathbb{H}^n$ (see Figure \ref{fig:visual}). Without loss of generality, we may assume $\tilde \Theta$ is the ideal point with infinite vertical coordinate, and that $H$ is the complement of the Euclidean ball of radius one centered at the origin.

Let $D$ be the Euclidean ball of radius one centered at the origin in $\partial_\infty\mathbb{H}^n$. That is, $D$ is the complement in the boundary at infinity of $\partial_\infty H$. It is easy to see that in the visual metric induced on $\partial_\infty\mathbb{H}^n$ by a point $z_t = (0,\ldots 0,t)$, $D$ is a ball of radius $r_t$ with $r_t$ strictly decreasing to  0 as $t\to \infty$.  Therefore we may take $t^*$ sufficiently large that for $z:=z_{t^*}$, $\nu_z(D)<\frac{1}{3}\Vert \nu_z\Vert$. Note that $t^*>1$.  Let $B_1'$ be the horoball centered at $\tilde \Theta$ whose boundary contains $z$; that is, $B_1'$ is all points with vertical coordinate $\geq t^*$. To complete the proof, we show that for any $y \in B_1'$, $\nu_y(D)<\frac{1}{3}\Vert \nu_y \Vert$.

Let $a = (0,\ldots, 0,1)$ and let $b$ be the closest point on $\partial H$ to $y$. Note that $d(z,a)\leq d(y,b)$. The hyperplane $\partial H$ is isometric to $\mathbb{H}^{n-1}$ and any isometry $g$ of $\partial H$ extends uniquely to an (orientation-preserving) isometry $\bar g$ of $\mathbb{H}^n$. Let $g$ be an isometry of $\partial H$ taking $b$ to $a$. Its extension $\bar g$ takes the geodesic ray perpendicular to $\partial H$ at $b$ to the perpendicular ray from $a$. Hence $\bar g(y)$ lies on the vertical axis, and since $d(z,a)\leq d(y,b) = d(\bar g(y), \bar g(b))$, $\bar g(y) = (0,\ldots,0,t)$ for some $t\geq t^*$. Note finally that $\bar g(D) = D$. Therefore, 
\[ \nu_y(D) = \nu_{\bar g(y)}(\bar g(D)) = \nu_{z_t}(D) \leq \nu_{z}(D) <\frac{1}{3}\Vert \nu_z\Vert = \frac{1}{3}\Vert \nu_y\Vert,\]
as desired.
\end{proof}

\tdplotsetmaincoords{60}{110}

\pgfmathsetmacro{\radius}{1}
\pgfmathsetmacro{\thetavec}{0}
\pgfmathsetmacro{\phivec}{0}

\begin{figure}
\begin{center}
\begin{tikzpicture}[scale=2,tdplot_main_coords]

\tdplotsetthetaplanecoords{\phivec}

\draw[dashed, fill=blue!50!white] (\radius,0,0) arc (0:360:\radius);
\shade (\radius,0,0) arc (0:360:\radius);

\draw (0,.5,.95) arc (280:265:2.4);
\draw[dashed] (-.2,.6,1.17) arc (265:240:2.4);
\draw (-1.5, .85, 1) node [circle, fill, inner sep=.03cm] {};
\draw (-1.5, .85, 1) node [anchor = west] {$y$};

\node at (.7,2.4,0) {$\partial_\infty\mathbb{H}^n$} ;
\node at (.7,2.5,2) {$\{ x_n=t^*\}$} ;
\node at (.5,-.3,0) {$D$} ;
\node at (-1,.85) {$\partial H$} ;

\draw (0,0,1) -- (0,0,1.2) ;
\draw[dashed] (0,0,1.2) -- (0,0,2) ;
\draw (0,0,2) -- (0,0,3.5) ;

\shade[ball color=blue!10!white,opacity=0.4] (1cm,0) arc (0:-180:1cm and 5mm) arc (180:0:1cm and 1cm);
\draw (0, 0, 1) node [circle, fill, inner sep=.03cm] () {};
\draw (0, 0, 1) node [anchor = east] () {$a$};
\draw (0, 0, 0) node [circle, fill, inner sep=.03cm] () {};
\draw (0, 0, 0) node [anchor=west] () {$\vec 0$};
\draw (0, 0, 2) node [circle, fill, inner sep=.03cm] () {};
\draw (0, 0, 2) node [anchor=east] () {$z$};
\draw (0, .5, .95) node [circle, fill, inner sep=.03cm] () {};
\draw (0, .5, .95) node [anchor=west] () {$b$};

\draw (2.25,-2,0) -- (-.75,-2,0) -- (-1.05,-1.2,0) ; 
\draw[dashed] (-1.05,-1.2,0) -- (-1.56,.2,0) ;
\draw (-1.56,.2,0) -- (-2.25,2,0) -- (.75,2,0) -- (2.25,-2,0);
\draw[dashed] (2.25,-2,2) -- (-.75,-2,2) -- (-2.25,2,2) -- (.75,2,2) -- (2.25,-2,2);

\draw (0, 0, 3) node [circle, fill, inner sep=.03cm] () {};
\draw (0, 0, 3) node [anchor=east] () {$\bar g y$};

\draw [thick, blue, ->] (0,.4,.95) arc (160:180:1) ;
\node[blue] at (.1,.2,.9) {$g$} ;

\draw [thick, blue, ->] (-1.7,.8,1) arc (120:162:3.4) ;
\node[blue] at (.3,1.1,3) {$\bar g$} ;

\end{tikzpicture}
\end{center}
\caption{The geometry of the proof of Lemma \ref{lem:control_visual}}\label{fig:visual}
\end{figure}
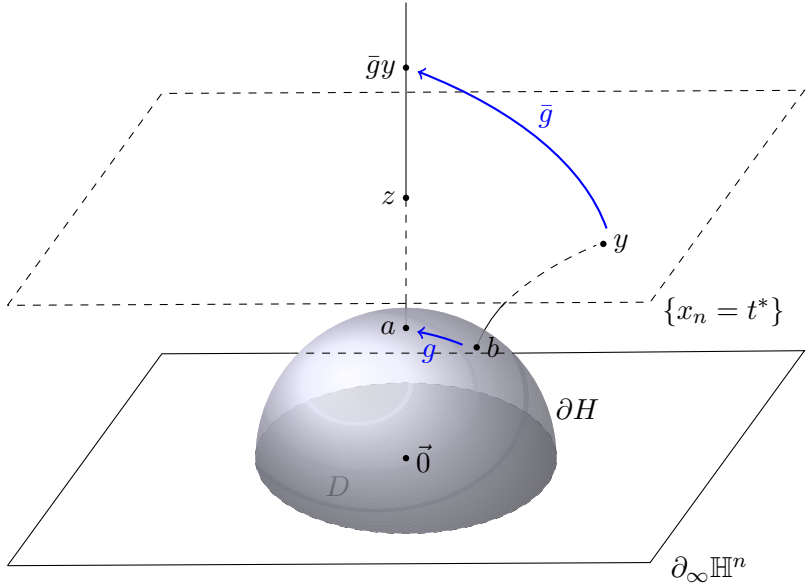

\begin{lemma}[Compare with Lemma 5.2 in \cite{BCS}]\label{lem:two thirds}
Let $H$ be any open half-space in $\tilde Y_0$ for which $\partial_\infty H$ contains the bounded parabolic point $\tilde\Theta$.  Then there exists a horoball $B_1 \subset \tilde Y_0$ based at $\tilde \Theta$ such that for all $x\in \tilde f^{-1}(B_1)$, $\tilde f_*\mu_x(\partial_\infty H) > \frac{2}{3}\Vert \mu_x \Vert$.
\end{lemma}

\begin{proof}
For any non-identity element $\gamma$ of the parabolic subgroup $\Gamma_{\tilde\Theta}$ stabilizing $\tilde
\Theta$, choose a fundamental domain $D_\gamma$ for the action of $\gamma$ on $\partial_\infty \tilde Y_0$. Then there exist integers $a_\gamma, k_\gamma$ such that 
\[ \partial_\infty \tilde Y_0 \setminus \partial_\infty H \subseteq \bigcup_{i=a_\gamma}^{a_\gamma+k_\gamma} \gamma^i (D_\gamma).\]

Then we can compute as follows:
\begin{multline*}
  \tilde f_*\mu_x\left( \partial_\infty \tilde Y_0 \setminus \partial_\infty H\right) \leq \sum_{i=0}^{k_\gamma} \tilde f_*\mu_x\left( \gamma^{i+a_\gamma}(D_\gamma) \right)  
  \leq \sum_{i=0}^{k_\gamma} e^{h d_\Omega(x,\gamma^i x)}\tilde f_*\mu_x\left(\gamma^{a_\gamma}
  (D_\gamma)\right) \\   \leq \sum_{i=0}^{k_\gamma} e^{i \cdot h  d_\Omega(x,\gamma x)}\tilde
  f_*\mu_x\left(\gamma^{a_\gamma} (D_\gamma)\right)  = \frac{1-e^{(k_\gamma+1) \cdot h
  d_\Omega(x,\gamma x)}}{1- e^{h d_\Omega(x,\gamma x)}}\tilde f_*\mu_x\left( \gamma^{a_\gamma}
  (D_\gamma) \right).
\end{multline*}
The second inequality comes from applying Lemma \ref{lem:PS rule} and using the $\Gamma$-equivariance of $\tilde f$. The third comes from a simple application of the triangle inequality.
On the other hand,
\begin{multline*}
	\Vert \tilde f_*\mu_x \Vert = \sum_{i\in\mathbb{Z}} \tilde f_*\mu_x\left(
	\gamma^{i+a_\gamma} (D_\gamma)\right) 
		  \geq \sum_{i\in\mathbb{Z}} e^{-h d_\Omega(x, \gamma^i x)} \tilde f_*\mu_x\left(\gamma^{a_\gamma} (D_\gamma)\right) \\
		  \geq \sum_{i\in\mathbb{Z}} e^{-|i| \cdot h d_\Omega(x, \gamma x)} \tilde f_*\mu_x\left( \gamma^{a_\gamma}(D_\gamma) \right) 
		  \geq \frac{1}{1-e^{-h d_\Omega(x, \gamma x)}} \tilde f_*\mu_x\left( \gamma^{a_\gamma} (D_\gamma)\right).
\end{multline*}
Again, the first inequality uses Lemma \ref{lem:PS rule} and the second the triangle inequality.
Therefore,
\begin{align*} 
	\frac{\tilde f_*\mu_x\left( \partial_\infty \tilde Y_0 \setminus \partial_\infty H\right)}{\Vert \tilde f_*\mu_x \Vert } &\leq \frac{1-e^{(k_\gamma+1) \cdot h d_\Omega(x,\gamma x)}}{1- e^{h  d_\Omega(x,\gamma x)}} 1-e^{h d_\Omega(x, \gamma x)} \\
		& = e^{k_\gamma h d_\Omega(x,\gamma x)} - e^{- h d_\Omega(x,\gamma x)}.
\end{align*}

Now choose $\epsilon_0>0$ so small that $e^{k_\gamma \cdot h \epsilon_0} - e^{- h \epsilon_0}<\frac{1}{3}.$ By Lemma \ref{lem:short loops in cusps} and the properness of $\tilde f$, we can choose a horoball $B_1$ centered at $\tilde \Theta$ such that for all $x\in \tilde f^{-1}(B_1)$, we have $d_\Omega(x, \gamma x)<\epsilon_0$. Therefore, for all such $x$ the argument above bounds $\frac{\tilde f_*\mu_x(\partial_\infty H)}{\Vert \tilde  \mu_x \Vert}$ as desired, since $\|\tilde f_\ast \mu_x\|=\|\mu_x\|$ by definition of the push-forward.
\end{proof}


%

\section{The natural map and its Jacobian}\label{sec:bcg}

We now turn to the natural map, the main tool of the Besson--Courtois--Gallot approach. The arguments in this section will only be sketched as they are standard adaptations of the barycenter method. Details specific to the Hilbert geometry setting can be found in \cite{abc} (or \cite{boland-newberger} for the Finsler manifold setting), and details on the method in general can be found in \cite{BCG1} with a survey in \cite{CFrecent}.

\begin{definition}\label{defn:natural map}
The \emph{natural map} is $\tilde \Phi \colon \tilde
Y_\Omega \to \tilde Y_0$ 
given by 
\[ \tilde \Phi(x) = bar(\tilde f_* \mu_x).\]
\end{definition}

It is easy to check that the natural map is $\Gamma$-equivariant, and so descends to a natural map $\Phi:Y_\Omega \to Y_0$. This map is used to compare the volumes of $Y_\Omega$ and $Y_0$ since its Jacobian can be bounded.

\begin{defn}\label{defn:distortion}
Let $g$ be any Riemannian metric on $Y_\Omega$. We define the \emph{eccentricity factor of
$F_\Omega$} with respect to $g$ as
\[ 
  N(F_\Omega, g) := \sup_{y\in Y_\Omega} \max_{v\in S_{g}(1,y)} \frac{F_\Omega(v)^n
    \Vol_{g}(B_{F_\Omega}(1,y))}{\Vol_{g}(B_{g}(1,y))} 
\]
where $B_-(1,y)$ is the ball of radius of 1 in the tangent space at $y$ with respect to the given norm.
\end{defn}

It is an easy exercise to see that $N(F_\Omega, g) \geq 1$.

\begin{rmk}
We note that, \emph{a priori}, $N(F_\Omega,g)$ may be infinite, since $Y_\Omega$ is non-compact. The statement of Theorem \ref{thm:main} holds in this case, but does so trivially and Corollary \ref{cor:uniform_volume_bound} does not follow in this case. In Section \ref{sec:main} we will show that there is in fact a $g$ such that $N(F_\Omega, g)<\infty$.
\end{rmk}

\begin{prop}\label{prop:Jacobian bound}
For any Riemannian metric $g$ on $Y_\Omega$, the Jacobian of $\tilde \Phi$ at any point $y \in \tilde Y_\Omega$ satisfies
  \[ 
    |Jac(\tilde \Phi)(y)| \leq \frac{h(F_\Omega)^n}{h(g_0)^n} N(F_\Omega, g).
  \]
If equality holds at any $y$, then $D_y\tilde \Phi\colon (T_y\tilde Y_\Omega, F_\Omega) \to (T_{\tilde \Phi(y)}\tilde Y_0, g_0)$ is an isometry composed with a homothety.
\end{prop}

Proposition \ref{prop:Jacobian bound} is proved in \S3.2 of \cite{abc}, following Boland and Newberger's \cite{boland-newberger} adaptation of \cite{BCG1}. The argument is conducted entirely at the level of the universal covers, so it works in any setting where the natural map can be defined and differentiated, without the requirement that the quotient be compact. 

In the argument in \S3.2 of \cite{abc}, it is noted that differentibility of the natural map hinges on differentiability of the Busemann functions $B_{x,\beta}^\Omega(-)$ (Definition \ref{def:busemannfunctions}). In the setting of strictly convex Hilbert geometries admitting a finite volume quotient, the boundary is $C^{1+\alpha}$ for some $\alpha>0$ at every point in $\partial\Omega$ \cite[Corollary 1.5]{cramponmarquis14}. Since the Finsler metric has the same regularity as the boundary, the Busemann functions $B_{x,\beta}^\Omega(y)$ are differentiable in $y$.

%

\section{The natural map is proper}\label{sec:proper}

To use the Jacobian bound given by \ref{prop:Jacobian bound} to compare the volumes of $Y_\Omega$ and $Y_0$, we need to know that $\Phi$ is proper. This was also the case for the extensions \cite{BCS} and \cite{storm} of entropy rigidity to the Riemannian, finite-volume setting. Our proof closely follows \cite{BCS}, avoiding some of the complications in \cite{storm} by relying on the particularly nice geometric properties of Hilbert geometries.

We prove that $\Phi$ is proper by proving that it is homotopic via a proper homotopy to the proper map $f$. It is easy to check that a map proper homotopic to a proper map is itself proper. The particular homotopy we use is as follows:

\begin{definition} \label{def:properhomotopy}
Let
\[ \tilde \Psi: [0,1] \times \tilde Y_\Omega \to \tilde Y_0\]
\[ (t,x) \mapsto bar\left(t\tilde f_*\mu_x + (1-t) \nu_{\tilde f(x)}\right).\]
\end{definition}

By Definition \ref{defn:natural map}, $\tilde \Psi(1,x)=\tilde \Phi(x)$, and by Lemma \ref{lem:bar visual}, $\tilde \Psi(0,x)=\tilde f(x)$.

\begin{lemma}
$\tilde \Psi$ is continuous.
\end{lemma}

\begin{proof}
First, $\tilde f$, $x\mapsto \mu_x$, and $f_*:\mathcal{M}(\partial_\infty\tilde Y_\Omega) \to \mathcal{M}(\partial_\infty \tilde Y_0)$ are continuous. Then using that $\nu_-:\tilde Y_0 \to \mathcal{M}(\partial_\infty \tilde Y_0)$ is continuous, $x \mapsto \nu_{\tilde f(x)}$ is continuous. Finally $bar:\mathcal{M}(\partial_\infty \tilde Y_0)\to \tilde Y_0$ is continuous. So $\tilde \Psi$ is continuous in $x$ for all $t\in [0,1]$. 

Continuity in $t$ follows from the continuity of $bar$. 
\end{proof}

It is clear that $\tilde \Psi$ is $\Gamma$-equivariant in its second variable, and so descends to a homotopy of maps between $Y_\Omega$ and $Y_0$. Our goal is to prove:

\begin{prop}\label{prop:proper hty}
$\Psi_t$ is a proper homotopy. (That is, it is proper as a map $[0,1]\times Y_\Omega \to Y_0$.)
\end{prop}

We approach the proof of Proposition \ref{prop:proper hty} via the following Lemma. The scheme of the proof is similar to the approaches to Theorem 3.1 and Proposition 5.1 in \cite{BCS}.

\begin{lemma}\label{lem:proper condition}
$\Psi_t$ is a proper homotopy if the following holds:
\begin{quote}
	For any cusp $\Theta$ in $Y_0$ and any neighborhood $U_0$ of $\Theta$, there exists a neighborhood $U_1$ of $\Theta$ such that if $\Psi_0(x)=f(x)\in U_1$, then $\Psi_t(x)\in U_0$ for all $t\in[0,1]$.
\end{quote}
\end{lemma}

This condition can be restated in the following way. Given any neighborhood $U_0$ of a cusp in $Y_0$, if a point $x$ is sent by $f$ sufficiently far into that cusp (i.e., into $U_1)$ then the entire track $\{\Psi_t(x)\colon t\in[0,1]\}$ of $x$ through the homotopy remains in $U_0$. (See Figure \ref{fig:proper condition}.)

\begin{figure}
\centering
\begin{subfigure}{.4\textwidth}
  \centering

\begin{tikzpicture}[scale=1.2] 

\draw [fill, blue!40] plot [smooth cycle, tension=.1] coordinates {(2,1.2) (1.5,1.08) (1,1) (.5,.95) (.2,.97) (0,1) (-.15,.7) (-.2,.4) (-.1,.05) (.1,-.18) (.24,-.28) (.5,-.03) (.7,.11) (1,.3) (1.5,.57) (2.1,.87)};
\draw [fill, purple!40] plot [smooth cycle, tension=.1] coordinates {(3,1.5) (2,1.2) (1.95,1) (2.1,.87) (3,1.3)};

\draw plot [smooth, tension=.7] coordinates {(3,1.3) (1,.3) (0,-.5) (-.5,-1)} ;

\draw plot [smooth, tension=.7] coordinates {(-1.5,1.7) (-.5,1.1) (1,1) (3,1.5)};

\draw[dashed] (0,1) arc (140:240:.85cm);
\draw[dashed] (2,1.2) arc (140:240:.2cm);

\node [rotate=115, yscale=2.5, xscale=1.2] at (2.7,.7) {\{};
\node at (3,.4) {$U_1$} ;

\node [rotate=295, yscale=7, xscale=1.2] at (1.3,1.8) {\{};
\node at (1.2,2.2) {$U_0$} ;

\node at (-1,1.8) {$Y_0$} ;

\draw[fill] (2.5,1.2) circle (.04);
\draw[fill] (.5,.7) circle (.04);

\draw plot [smooth, tension=.7] coordinates {(2.5,1.2) (2,1.1) (1.5,.8) (1,.9) (.5,.7)};

\node at (2.5,1.7) {$\tilde f(x)$} ;
\node at (.5,.4) {$\Phi(x)$} ;

\end{tikzpicture}
\caption{The properness condition of Lemma \ref{lem:proper condition}.}\label{fig:proper condition}
\end{subfigure} \qquad \quad
\begin{subfigure}{.4\textwidth}
  \centering
\begin{tikzpicture}[scale=.6] 

\draw (0,0) circle (3.5);
\draw[fill=blue!40] (1,0) circle (2.5);
\draw[fill=purple!40] (2.8,0) circle (.7);

\draw[pattern = north east lines] (2.65,2.3) arc (130:230:3cm);
\draw[pattern = north east lines] (2.65,2.3) arc (41:-41:3.5cm);
\draw(2.1,2.8) arc (130:230:3.65cm);

\draw (1,0) circle (2.5);

\draw[fill] (3.5,0) circle (.05);
\draw[thick, purple, ->] (4,-.9) -- (3.5,-.6) ;

\node at (-3,2.8) {$\tilde Y_0$};
\node at (-1,-2.1) {$B_0$};
\node[purple] at (4.4,-1) {$B_1$};

\node at (4,0.1) {$\tilde\Theta$};

\node [rotate=90, yscale=1, xscale=.8] at (1.2,-.2) {\{};

\node at (2.7,3) {$H_0$};
\node at (3.5,2.1) {$H_1$};
\node[scale=.5] at (1.25,-.5) {$\smash{>}D$};

\end{tikzpicture}
\caption{The setup for the proof of Proposition \ref{prop:proper hty}.}\label{fig:proper hty}\end{subfigure}
\end{figure}

\begin{proof}[Proof of Lemma \ref{lem:proper condition}]
Suppose that the condition given in the statement of the lemma holds. Let $K\subset Y_0$ be compact. We want to show that $\Psi^{-1}(K) \subset [0,1]\times Y_\Omega$ is compact. Without loss of generality, we can assume that $K$ is of the form $Y_0 \setminus \left( \bigcup_{i=1}^n U_0^{(i)}\right)$, where $U_0^{(i)}$ is a neighborhood of the $i^{th}$ cusp of $Y_0$.

For each $i$, pick $U_1^{(i)} \subset U_0^{(i)}$ as described in the statement of the Lemma. $Y_0 \setminus \left( \bigcup_{i=1}^n U_1^{(i)}\right)$ is compact, and $f$ is a proper map, so $f^{-1}\left( Y_0 \setminus \left( \bigcup_{i=1}^n U_1^{(i)}\right)\right)$ is compact.

Now suppose that $\Psi_t(x) \in K = Y_0 \setminus \left( \bigcup_{i=1}^n U_0^{(i)}\right)$. Then by the condition of the Lemma, $f(x) \notin \bigcup_{i=1}^n U_1^{(i)}$ and so $x\in f^{-1}\left( Y_0 \setminus \left( \bigcup_{i=1}^n U_1^{(i)}\right)\right)$, a compact set. $\Psi^{-1}(K)$ is closed, so $\Psi^{-1}(K)$ is a closed subset of $[0,1]\times f^{-1}\left( Y_0 \setminus \left( \bigcup_{i=1}^n U_1^{(i)}\right)\right)$. Therefore, $\Psi^{-1}(K)$ is compact, as desired.
\end{proof}

\begin{proof}[Proof of Proposition \ref{prop:proper hty}]

We prove the condition of Lemma \ref{lem:proper condition}. Lift a cusp $\Theta$ in $Y_0$ to a bounded parabolic point $\tilde\Theta \in \partial_\infty \tilde Y_0$. Then there exists a horoball $B_0$ in $\tilde Y_0$ centered at $\tilde \Theta$ which is contained in a lift $\tilde U_0$ of $U_0$ around $\tilde \Theta$.

Let $H_0$ be a halfspace in $\tilde Y_0$ so that $\partial_\infty H_0$ contains $\tilde \Theta$ and $\tilde Y_0 \setminus H_0$ contains a fundamental domain for the action of $\Gamma_{\tilde\Theta}$ on the horosphere $\partial B_0$. Let $D$ be the constant provided by Lemma \ref{lem:control_bar}.  Let $H_1$ be a second halfspace, chosen so that $H_1$ contains $\tilde\Theta$ and the $D$-neighborhood of $H_1$ is contained in $H_0$. Using Lemma \ref{lem:two thirds}, let $B_1$ be a horoball based at $\tilde\Theta$ inside $H_1$ such that for all $x\in \tilde f^{-1}(B_1)$, 
\[ \tilde f_* \mu_x(\partial_\infty H_1) > \frac{2}{3}\Vert \mu_x\Vert. \]
By Lemma \ref{lem:control_visual}, we can shrink $B_1$ if needed so that for all $y\in B_1$, we have moreover that $\nu_y(\partial_\infty H_1) >\frac{2}{3}.$  (See Figure \ref{fig:proper hty}.) Let $U_1$ be the projection of $B_1$ to $Y_0$.

Now suppose that $\tilde f(x)\in B_1$, i.e., $x$ projects to a point in $Y_\Omega$ which maps to $U_1$ under $f$. Then $\tilde f(x)\in B_1$, $\tilde f_*\mu_x(\partial_\infty H)>\frac{2}{3}\Vert \mu_x\Vert$, and $\nu_{\tilde f(x)}(\partial_\infty H) >\frac{2}{3}$. Therefore, for all $t\in[0,1]$, 
\begin{multline}
	(t \tilde f_*\mu_x + (1-t)\nu_{\tilde f(x)})(\partial_\infty H)   > t (\frac{2}{3}\Vert \mu_x\Vert) + (1-t) \frac{2}{3} \nonumber  \\
							  = \frac{2}{3}(t \Vert \mu_x \Vert + (1-t) \Vert \nu_{\tilde f(x)}\Vert ) 
							  = \frac{2}{3} \Vert t \tilde f_*\mu_x +(1-t) \nu_{\tilde f(x)} \Vert. \nonumber
\end{multline}

By Lemma \ref{lem:control_bar}, for all $t\in [0,1]$, $d_{g_0}(bar(t \tilde f_*\mu_x + (1-t)\nu_{\tilde f(x)}), H_1)<D.$ Therefore, by the choice of $H_1$, for all $t$, $\Psi_t(x) \in H_0$. Thus,  $\Psi_t(x)\in U_0$ for all $x\in f^{-1}(U_1)$, proving the condition of Lemma \ref{lem:proper condition} as desired.
\end{proof}

%

\section{Proof of the main theorem}\label{sec:main}

We are now ready to prove our main theorem, but before doing so we need a Riemannian metric on $Y_\Omega$ so that $N(F_\Omega, g)<\infty$. We have such a choice in the \emph{Blaschke metric}, also known as the {\em affine metric} (see for instance \cite[Definition 2.2]{benoisthulin14} and \cite{chengyau77,chengyau86}). This metric exists for any properly convex $\Omega$ and is easily seen to be projectively invariant, so it descends to $Y_\Omega$. Moreover, we have the following uniform comparison between the Hilbert and Blaschke metrics observed by Benoist and Hulin:

\begin{proposition}[{\cite[Proposition 3.4]{benoisthulin14}}] \label{prop:blaschke}
Given any properly convex domain $\Omega$ in $\mathbb {RP}^n$, there exists a constant $K_n\geq 1$ depending only on $n$ such that for all $v\in T\Omega$,
\[ \frac1{K_n}A_\Omega(v)\leq F_\Omega(v)\leq K_n A_\Omega(v) \]
where $A_\Omega$ is the norm defined by the Blaschke metric.
\end{proposition}

\begin{remark}\label{rem:comparingmetrics}
In fact, Proposition \ref{prop:blaschke} is true for any natural projectively invariant norm defined for properly convex open sets, by a cocompactness argument following a theorem of Benzecri \cite{benzecri_these}.  See also \cite[\S 9, Prop 9.7]{marquis_handbook}.
\end{remark}

\begin{remark}\label{rem:relatingthemetrics}
Under the assumption that $\Omega$ admits a discrete action by a noncompact group $\Gamma$ of projective transformations, the Blaschke and Hilbert metrics agree if and only if $\Omega$ is an ellipsoid. Since the fundamental group of a finite volume manifold is not compact, this applies to our setting.

To see this, first assume $\Omega$ is an ellipsoid. Then $\Omega$ admits a transitive action by a group of projective transformations, which are isometries for both the Blaschke metric and the Hilbert metric, hence the metrics agree in this case.

Conversely, if the metrics agree, then the Hilbert metric is a {\em regular Finsler metric}, meaning the norm is $C^2$ with positive definite Hessian, since the Blaschke metric is in fact analytic and positive definite (\cite{chengyau77}, or see \cite[\S 1.2]{tholozan17} for the statements in less generality which is relevant here). If the Hilbert norm is a regular Finsler norm then either $\Omega$ is an ellipsoid or has a compact isometry group (\cite{sociemethou}, see also \cite[Theorem 2.2]{crampon_handbook}). By assumption, it follows that $\Omega$ is an ellipsoid. 
\end{remark}

\begin{lemma}
Let $g_\Omega$ be a family of projectively invariant Riemannian metrics on properly convex domains $\Omega$ in $\mathbb {RP}^n$. Then there is a constant $K_n(g_\Omega)$ such that for all properly convex domains $\Omega$, 
  \[
    K_n^{-2n}\leq N(F_\Omega,g_\Omega) \leq K_n^{2n}
  \]
where $N(F_\Omega,g_\Omega)$ is the eccentricity factor from Definition \ref{defn:distortion}. 
  \label{lem:eccentricities_are_bounded}
\end{lemma}

\begin{proof}
For the proof, we suppress the subscripts and let $F=F_\Omega$ denote the Finsler Hilbert norm on $\Omega$ and $g=g_\Omega$ the Riemannian metric. Since the metric $g$ is projectively invariant, let $K=K_n(g)$ be the uniform constant comparing these metrics given in Remark \ref{rem:comparingmetrics}. Then for all $y\in Y_\Omega$ and $v\in S_g(1,y)$, since $g$ is Riemannian, 
  \begin{multline*}
    K^{-2n}\leq \frac{(K^{-1})^n\Vol_g(B(K^{-1},y))}{\Vol_g(B(1,y))}\leq 
    \frac{F(v)^n\Vol_g(B_F(1,y))}{\Vol_g(B_g(1,y))} 
    \\
    \leq \frac{K^n\Vol_g(B_g(K,y))}{\Vol_g(1,y)}\leq K^{2n}.
  \end{multline*}
The result follows. 
\end{proof}

Note that by Remark \ref{rem:relatingthemetrics}, we can choose $K_n(A_\Omega)=1$ where $A_\Omega$ is the Blaschke metric if and only if $\Omega$ is an ellipsoid, and hence $N_\Omega=1$ if and only if $(\tilde Y_\Omega,F_\Omega)$ and $(\tilde Y_0,g_0)$ are already isometric.

We now prove our main theorem and its corollary:

\begin{thm}\label{thm:rigidity}
Let $Y_\Omega$ be a finite volume convex projective manifold of dimension $n\geq 3$, equipped with its Hilbert metric. Suppose that $Y_0$ is a hyperbolic structure on the same manifold. Then 
\[  N_\Omega h(F_\Omega)^n \Vol(Y, F_\Omega) \geq h(g_0)^n \Vol(Y, g_0) \]
where $N_\Omega:=N_\Omega(F_\Omega, A_\Omega)\geq 1$ is the eccentricity factor of the Hilbert metric relative to the Blaschke metric. 

Furthermore, equality holds if and only if $(Y, F_\Omega)$ is isometric to $(Y, g_0)$.
\end{thm}

\begin{proof}
Proposition \ref{prop:proper hty}, together with the properness of $f$ immediately implies that the natural map $\Phi$ is proper. The fact that $\Phi$ is proper then allows us to compare the volumes of $Y_\Omega$ and $Y_0$ by integrating over compact exhaustions of these space and using the Jacobian bound of Proposition \ref{prop:Jacobian bound} to prove the desired inequality.

If equality holds, then the equality case of Proposition \ref{prop:Jacobian bound} must hold at almost all points. That it holds at a single point $x$ in $\tilde Y_\Omega$ tells us that the unit sphere in $T_x\tilde Y_\Omega$ is an ellipsoid, and in particular is $C^2$ with positive definite Hessian. Then the Hilbert norm is $C^2$ with positive definite Hessian, meaning it is a regular Finsler norm.  As in Remark \ref{rem:relatingthemetrics}, if the Hilbert norm on a properly convex domain $\Omega$ is $C^2$, then either $\Omega$ is an ellipsoid or has a compact isometry group (\cite{sociemethou}, see also \cite[Theorem 2.2]{crampon_handbook}). Since the fundamental group of the quotient is not compact and acts by isometries on $\tilde Y_\Omega=\Omega$, we conclude $\Omega$ must be an ellipsoid. Then the Mostow-Prasad Rigidity Theorem \cite{prasad} implies that $(Y_\Omega, F_\Omega)$ and $(Y_0,g_0)$ are isometric. 
\end{proof}

\begin{proof}[Proof of Corollary \ref{cor:uniform_volume_bound}]
By Theorem \ref{thm:rigidity}, 
  \[
    \Vol(Y,F_\Omega)\geq \left( \frac{h(g_0)}{h(F_\Omega)} \right)^n N_\Omega^{-1} \Vol(Y,g_0).
  \]
The volume growth entropy satisfies the inequality $h(F_\Omega)\leq n-1=h(g_0)$
  \cite[Theorem 2]{tholozan17}, so by Lemma \ref{lem:eccentricities_are_bounded},
  \[
    \Vol(Y,F_\Omega)\geq N_\Omega^{-1} \Vol(Y,g_0)\geq K_n^{-2n}\Vol(Y,g_0). 
  \]
Taking $\mathcal{D}=K_n^{-2n}$ finishes the proof.
\end{proof}

\bibliographystyle{alpha}
\bibliography{biblio}

\begin{thebibliography}{BCG96}

\bibitem[ABC19]{abc}
Ilesanmi Adeboye, Harrison Bray, and David Constantine.
\newblock Entropy rigidity and {H}ilbert volume.
\newblock {\em Discrete \& Continuous Dynamical Systems - A}, 39(4):1731--1744,
  2019.

\bibitem[BC20]{ballascasella}
Samuel~A. Ballas and Alex Casella.
\newblock Gluing equations for real projective structures on 3-manifolds.
\newblock 2020.
\newblock preprint, \url{https://arxiv.org/abs/1912.12508}.

\bibitem[BCG95]{BCG1}
G\'erard Besson, Gilles Courtois, and Sylvestre Gallot.
\newblock Entropies et rigidit\'es des espaces localement sym\'etriques de
  courbure strictement n\'egative.
\newblock {\em Geom. Funct. Anal.}, 5(5):731--799, 1995.

\bibitem[BCG96]{BCG2}
G\'erard Besson, Gilles Courtois, and Sylvestre Gallot.
\newblock Minimal entropy and {M}ostow's rigidity theorems.
\newblock {\em Ergodic Theory and Dynamical Systems}, 16(4):623--649, 1996.

\bibitem[BCS05]{BCS}
Jeffrey Boland, Chris Connell, and Juan Souto.
\newblock Volume rigidity for finite volume manifolds.
\newblock {\em American Journal of Mathematics}, 127(3):535--550, 2005.

\bibitem[Ben60]{benzecri_these}
Jean~Paul Benz\'ecri.
\newblock Sur les vari\'et\'es localement affines et localement projectives.
\newblock 88:229--332, 1960.
\newblock French.

\bibitem[Ben06]{Ben5}
Yves Benoist.
\newblock Convexes hyperboliques et quasiisom\'etries.
\newblock {\em Geom. Dedicata}, 122:109--134, 2006.

\bibitem[BH14]{benoisthulin14}
Yves Benoist and Dominique Hulin.
\newblock Cubic differentials and hyperbolic convex sets.
\newblock {\em Journal of Differential Geometry}, 98(1):1--19, 2014.

\bibitem[BM16]{ballasmarquis}
Samuel~A Ballas and Ludovic Marquis.
\newblock Properly convex bending of hyperbolic manifolds.
\newblock 10/9/2016.
\newblock preprint, to appear in Groups, Geom, and Dynam.

\bibitem[BN01]{boland-newberger}
Jeff Boland and Florence Newberger.
\newblock Minimal entropy rigidity for {F}insler manifolds of negative flag
  curvature.
\newblock {\em Ergodic Theory and Dynamical Systems}, 21(1):13--23, 2001.

\bibitem[CF03]{CFrecent}
Christopher Connell and Benson Farb.
\newblock Some recent applications of the barycenter method in geometry.
\newblock In {\em Topology and Geometry of Manifolds}, volume~71 of {\em Proc.
  Sympos. Pure Math.}, pages 19--50. American Math Society, 2003.

\bibitem[CLT15]{clt}
D~Cooper, D.D Long, and S~Tillmann.
\newblock On convex projective manifolds and cusps.
\newblock {\em Advances in Mathematics}, 277:181--251, 2015.

\bibitem[CM14a]{cramponmarquis_finitude}
Micka\"{e}l Crampon and Ludovic Marquis.
\newblock Finitude g\'{e}om\'{e}trique en g\'{e}om\'{e}trie de {H}ilbert.
\newblock {\em Ann. Inst. Fourier (Grenoble)}, 64(6):2299--2377, 2014.

\bibitem[CM14b]{cramponmarquis14}
Micka{\"e}l Crampon and Ludovic Marquis.
\newblock Le flot g\'eod\'esique des quotients g\'eom\'etriquement finis des
  g\'eom\'etries de {H}ilbert.
\newblock {\em Pacific J. Math.}, 268(2):313--369, 2014.

\bibitem[Cra11]{cramponthese}
Micka{\"e}l Crampon.
\newblock {\em Dynamics and entropies of {H}ilbert metrics}.
\newblock Institut de Recherche Math\'ematique Avanc\'ee, Universit\'e de
  Strasbourg, Strasbourg, 2011.
\newblock Th{\`e}se, Universit{\'e} de Strasbourg, Strasbourg, 2011.

\bibitem[Cra14]{crampon_handbook}
Micka\"{e}l Crampon.
\newblock The geodesic flow of {F}insler and {H}ilbert geometries.
\newblock In {\em Handbook of {H}ilbert geometry}, volume~22 of {\em IRMA Lect.
  Math. Theor. Phys.}, pages 161--206. Eur. Math. Soc., Z\"{u}rich, 2014.

\bibitem[CY77]{chengyau77}
Shiu~Yuen Cheng and Shing~Tung Yau.
\newblock On the regularity of the {M}onge-{A}mp\`ere equation {${\rm
  det}(\partial ^{2}u/\partial x_{i}\partial sx_{j})=F(x,u)$}.
\newblock {\em Comm. Pure Appl. Math.}, 30(1):41--68, 1977.

\bibitem[CY86]{chengyau86}
Shiu~Yuen Cheng and Shing-Tung Yau.
\newblock Complete affine hypersurfaces. {I}. {T}he completeness of affine
  metrics.
\newblock {\em Comm. Pure Appl. Math.}, 39(6):839--866, 1986.

\bibitem[JM87]{johnsonmillson}
Dennis Johnson and John~J. Millson.
\newblock Deformation spaces associated to compact hyperbolic manifolds.
\newblock In {\em Discrete groups in geometry and analysis ({N}ew {H}aven,
  {C}onn., 1984)}, volume~67 of {\em Progr. Math.}, pages 48--106. Birkh\"auser
  Boston, Boston, MA, 1987.

\bibitem[Kap07]{Kapovich2007}
Michael Kapovich.
\newblock Convex projective structures on {G}romov-{T}hurston manifolds.
\newblock {\em Geometry \& Topology}, 11:1777--1830, 2007.

\bibitem[Mar12]{marquis12}
Ludovic Marquis.
\newblock {Exemples de vari{\'e}t{\'e}s projectives strictement convexes de
  volume fini en dimension quelconque}.
\newblock {\em {L'Enseignement Math{\'e}matique}}, Tome 58:p. 3--47, 2012.
\newblock L'enseignement math{\'e}matique (2) 58 (2012) p3-47.

\bibitem[Mar14]{marquis_handbook}
Ludovic Marquis.
\newblock Around groups in {H}ilbert geometry.
\newblock In {\em Handbook of {H}ilbert geometry}, volume~22 of {\em IRMA Lect.
  Math. Theor. Phys.}, pages 207--261. Eur. Math. Soc., Z\"urich, 2014.

\bibitem[Nie15]{nie}
Xin Nie.
\newblock On the {H}ilbert geometry of simplicial {T}its sets.
\newblock {\em Ann. Inst. Fourier (Grenoble)}, 65(3):1005--1030, 2015.

\bibitem[Pra73]{prasad}
Gopal Prasad.
\newblock Strong rigidity of {Q}-rank 1 lattices.
\newblock {\em Inventiones Mathematicae}, 21(4):255--286, 1973.

\bibitem[PT14]{Hilberthandbook}
A.~Papadopoulos and M.~Troyanov, editors.
\newblock {\em Handbook of {H}ilbert geometry}, volume~22 of {\em IRMA Lectures
  in Mathematics and Theoretical Physics}.
\newblock European Mathematical Society, 2014.

\bibitem[SM02]{sociemethou}
E~Socie-Methou.
\newblock Caract\'erisation des ellipso\"ides par leurs groupes
  d'automorphismes.
\newblock {\em Annales Scientifiques de l'\'Ecole Normale Sup{\'e}rieure},
  35(4):537--548, 2002.

\bibitem[Sto06]{storm}
P.A. Storm.
\newblock The minimal entropy conjecture for nonuniform rank one lattices.
\newblock {\em Geom. Funct. Anal.}, 16:959--980, 2006.

\bibitem[Tho17]{tholozan17}
Nicolas Tholozan.
\newblock Volume entropy of {H}ilbert metrics and length spectrum of {H}itchin
  representations into {${\rm PSL}(3,\Bbb{R})$}.
\newblock {\em Duke Math. J.}, 166(7):1377--1403, 2017.

\bibitem[Zha15]{zhang_degeneration}
Tengren Zhang.
\newblock The degeneration of convex $\mathbb{RP}^2$ structures on surfaces.
\newblock {\em Proc. Lond. Math. Soc.}, 111(5):967--1012, 2015.

\end{thebibliography}

\end{document}